\documentclass[leqno]{lms}

\usepackage{amsmath}
\usepackage{amssymb}
\usepackage[all]{xy}
\usepackage{csquotes}

\newtheorem{theorem}{Theorem}[section]
\newtheorem{lemma}[theorem]{Lemma}
\newtheorem{corollary}[theorem]{Corollary}
\newtheorem{proposition}[theorem]{Proposition}
\newunnumbered{ithm}{Theorem}

\newcommand{\acat}[1]{\mathsf{a}_{#1}}
\newcommand{\cacat}[1]{\widehat{\mathsf{a}}_{#1}}
\newcommand{\cat}{\mathcal{A}}

\newcommand{\scat}{\mathbf{\cat}(\sfam)}
\newcommand{\aff}{\mathbb{A}}

\newcommand{\defm}[1]{\mathsf{Def}_{#1}}

\newcommand{\fm}[1]{\mathsf{#1}}
\newcommand{\sfam}{\fm{S}}

\newcommand{\infdef}{\mathrm{T}^1}

\newcommand{\sets}{\mathsf{Sets}}

\newcommand{\grdhol}{\mathbf{grHol}_{D}}

\DeclareMathOperator{\coker}{coker}

\DeclareMathOperator{\enm}{End}
\DeclareMathOperator{\ext}{Ext}

\DeclareMathOperator{\hmm}{Hom}

\DeclareMathOperator{\im}{im}
\DeclareMathOperator{\mor}{Mor}
\DeclareMathOperator*{\osum}{\oplus}

\DeclareMathOperator{\spec}{Spec}

\DeclareMathOperator{\diff}{Diff}
\DeclareMathOperator{\re}{Re}

\newcommand{\cc}{\mathbb{C}}

\newcommand{\extn}{\mathbf{Ext}(\sfam)}
\newcommand{\extncl}[1]{\mathbf{Ext}(\sfam,{#1})}
\newcommand{\extniso}[1]{\mathcal{E}(\sfam,{#1})}

\newcommand{\mcat}[1]{{}\mathbf{Mod}_{#1}}

\newcommand{\n}{\mathbf{N}}

\newcommand{\uc}{\textbf{UC}}
\newcommand{\vect}[1]{\underline{#1}}
\newcommand{\weyl}{A_1(k)}

\newcommand{\z}{\mathbf{Z}}

\title[Iterated extensions]{Iterated Extensions and Uniserial Length Categories}
\author{Eivind Eriksen}
\classno{18E10 (primary), 16G99, 14F10 (secondary)}

\begin{document}

\maketitle

\begin{abstract}
In this paper, we study length categories using iterated extensions. We
consider the problem of classifying all indecomposable objects in a length
category, and the problem of characterizing those length categories that are
uniserial. We solve the last problem, and obtain a necessary and sufficient
criterion for uniseriality under weak assumptions. This criterion turns out
to be known by Amdal and Ringdal already in 1968; we give a new proof that is
both elementary and constructive. The first problem is the most fundamental 
one, and its general solution is \enquote{the main and perhaps hopeless 
purpose of representation theory} according to Gabriel. We solve the problem 
in the case when the length category is uniserial, using our constructive 
methods. As an application, we classify all graded holonomic $D$-modules on a 
monomial curve over the complex numbers, obtaining the most explicit results 
over the affine line, when $D$ is the first Weyl algebra. Finally, we show 
that the iterated extensions are completely determined by the noncommutative 
deformations of its simple factors. This tells us precisely what we can learn 
about a length category by studying its species; it gives the tangent space 
of the noncommutative deformation functor, or the infinitesimal deformations, 
but not the obstructions for lifting these deformations.
\end{abstract}

\section*{Introduction}

Let $\sfam = \{ S_{\alpha}: \alpha \in I \}$ be a family of non-zero, pairwise
non-isomorphic objects in an Abelian $k$-category $\cat$, where $k$ is a
field. We consider the minimal full subcategory $\scat \subseteq \cat$ that
contains $\sfam$ and is closed under extensions. The family $\sfam$ is called
a family of orthogonal points if $\enm(S_{\alpha})$ is a division algebra and
$\hmm(S_{\alpha},S_{\beta}) = 0$ for all $\alpha, \beta \in I$ with $\alpha
\neq \beta$. In this case, $\scat \subseteq \cat$ is a length category with
$\sfam$ as its simple objects.

We use the category $\extn$ of \emph{iterated extensions} of $\sfam$ to study
the length category $\scat$. An iterated extension of $\sfam$ is a couple
$(X,C)$ where $X$ is an object in $\cat$, and $C$ is a cofiltration
    \[ X = C_n \xrightarrow{f_n} C_{n-1} \to \dots \to C_2
    \xrightarrow{f_2} C_1 \xrightarrow{f_1} C_0 = 0 \]
where $f_i: C_i \to C_{i-1}$ is surjective and $K_i = \ker(f_i) \cong
S_{\alpha(i)}$ with $\alpha(i) \in I$ for $1 \le i \le n$. Hence the assignment
$(X,C) \mapsto X$ defines a forgetful functor $\extn \to \scat$. When we work
with the category $\extn$ of iterated extensions, the \emph{order vector}
$\vect \alpha = ( \alpha(1), \dots, \alpha(n) ) \in I^n$ is an invariant, in
addition to the usual invariants in the length category $\scat$ such as the
length $n$, the simple factors $\{ K_1, \dots, K_n \}$, and their
multiplicities.

An important special case is when $\cat = \mcat A$ is the category of modules
over an associative $k$-algebra $A$, and $\sfam$ is a subset of the simple
$A$-modules. If $\sfam$ is the family of all simple modules, then $\scat$ is
the category of all modules of finite length. There are also many other
interesting applications, for example when $\cat$ is the category of graded
modules over a graded $k$-algebra, or the category of coherent sheaves over
a $k$-scheme. Note that any length category is exact equivalent to an exact
subcategory of a module category. Nevertheless, it is often better to work
directly in the Abelian category of interest than to use such an embedding.

We say that $\scat$ is a uniserial length category if any indecomposable
object in $\scat$ has a unique composition series, and that at point $S$ in
$\sfam$ is $k$-rational if $\enm(S) = k$. When $\sfam$ is a family of
$k$-rational orthogonal points, we show that $\scat$ is a uniserial length
category if and only if the family $\sfam$ satisfies the condition
\begin{align*} \tag{\uc}
\sum_{\beta \in I} \, \dim_k \ext^1_{\cat}(S_{\alpha},S_{\beta}) \le 1 &
\quad \text{for all $\alpha \in I$} \\
\sum_{\alpha \in I} \, \dim_k \ext^1_{\cat}(S_{\alpha},S_{\beta}) \le 1 &
\quad \text{for all $\beta \in I$}
\end{align*}
It turns out that the condition (\uc) and the characterization of uniserial
length categories was known already in the 60's; see Section 8.3 in Gabriel
\cite{gabr73}. As far as we know, it first appeared in Amdal, Ringdal
\cite{amda-ring68}, where it is stated without proof. We give an elementary 
and constructive proof of the result that $\scat$ is uniserial if and only 
if (\uc) is satisfied, using the category $\extn$ of iterated extensions. 
In fact, after showing that the condition is necessary, we explicitly 
construct all indecomposable objects in $\scat$ when (\uc) holds, and prove 
that these objects are uniserial.

\begin{ithm}
Let $\sfam = \{ S_{\alpha}: \alpha \in I \}$ be a family of orthogonal
$k$-rational points in an Abelian $k$-category $\cat$. If $\sfam$ satisfies
(\uc), then the indecomposable objects in $\scat$ of length $n$ are given by
$\{ X(\vect \alpha): \vect \alpha \in \mathcal J \}$, up to isomorphism in
$\scat$, where the subset $\mathcal J \subseteq I^n$ consists of the vectors
$\vect \alpha$ such that the following conditions hold:
\begin{enumerate}
\item
$\ext^1_{\cat}(S_{\alpha(i-1)},S_{\alpha(i)}) \neq 0$ for $2 \le i \le n$
\item
If $\sigma_i \in \ext^1_{\cat}(S_{\alpha(i-1)},S_{\alpha(i)})$ is non-zero
for $2 \le i \le n$, then the matric Massey product $\langle \sigma_2,
\sigma_3, \dots, \sigma_n \rangle$ is defined and contains zero.
\end{enumerate}
Moreover, the indecomposable objects $X(\vect \alpha)$ are uniserial, and
can be constructed from the family $\sfam$ and their extensions.
\end{ithm}

As an application, we show that the category $\grdhol$ of graded holonomic
$D$-modules is uniserial when $D = \diff(A)$ is the ring of differential
operators on a monomial curve $A$ defined over the field $k = \cc$ of complex
numbers. Moreover, we classify all indecomposable objects in $\grdhol$. We
build upon the results in Eriksen \cite{erik18}, where we studied this
category. We obtain the most explicit result in the case when $A = k[t]$ and
$D = \weyl$ is the first Weyl algebra. The classification is similar in the 
other cases, since all rings of differential operators on monomial curves are 
Morita equivalent.

\begin{ithm}
Let $D = \weyl$ be the first Weyl algebra. Then the category $\grdhol$ of
graded holonomic $D$-modules is uniserial, and the indecomposable $D$-modules
in $\grdhol$ are, up to graded isomorphisms and twists, given by
    \[ M(\alpha,n) = D/D \, (E-\alpha)^n, \quad M(\beta,n) = D/D \; \mathbf
    w(\beta,n) \]
where $n \ge 1$, $\alpha \in J^* = \{ \alpha \in k: 0 \le \re(\alpha) < 1,
\; \alpha \neq 0 \}$, $\beta \in \{ 0, \infty \}$, and $\mathbf
w(\beta,n)$ is the alternating word on $n$ letters in $t$ and $\partial$,
ending with $\partial$ if $\beta = 0$, and in $t$ if $\beta = \infty$.
\end{ithm}

In the last section, we prove that for a swarm $\sfam$ of orthogonal points
in an Abelian $k$-category $\cat$, the iterated extensions of the family
$\sfam$ are completely determined by the noncommutative deformations of its
simple factors. Hence the length category $\scat$ is also determined by these
deformations. If the noncommutative deformations are unobstructed, then they
are determined by the species of $\scat$. This is the case for modules over a
hereditary ring, such as the ring $D$ of differential operators on a monomial
curve over the complex numbers. In general, we need both the species of
$\scat$, which defines the noncommutative deformations on the tangent level,
and the obstructions for lifting these deformations, to determine the
iterated extensions in $\extn$.

\section{Iterated extensions}
\label{s:itextn}

Let $k$ be a field, let $\cat$ be an Abelian $k$-category, and let $\sfam
= \{ S_{\alpha}: \alpha \in I \}$ be a fixed family of non-zero, pairwise
non-isomorphic objects in $\cat$. In this section, we define the category
$\extn$ of \emph{iterated extensions} of the family $\sfam$, equipped with
a forgetful functor $\extn \to \scat$ into the minimal full subcategory
$\scat \subseteq \cat$ that contains $\sfam$ and is closed under
extensions, and study its properties.

An object of $\extn$ is a couple $(X,C)$, where $X$ is an object of the
category $\cat$ and $C$ is a cofiltration of $X$ in $\cat$ of the form
    \[ X = C_n \xrightarrow{f_n} C_{n-1} \to \dots \to C_2
    \xrightarrow{f_2} C_1 \xrightarrow{f_1} C_0 = 0 \]
where $f_i: C_i \to C_{i-1}$ is surjective and $K_i = \ker(f_i) \cong
S_{\alpha(i)}$ with $\alpha(i) \in I$ for $1 \le i \le n$. The integer $n
\ge 0$ is called the \emph{length}, the objects $K_1, \dots, K_n$ are called
the \emph{factors}, and the vector $\underline \alpha = ( \alpha_1, \dots,
\alpha_n)$ is called the \emph{order vector} of the iterated extension $(X,
C)$.

Let $(X,C)$ and $(X',C')$ be a pair of objects in $\extn$ of lengths $n,n'
\ge 0$. A morphism $\phi: (X,C) \to (X',C')$ in $\extn$ is a collection
$\{ \phi_i: 0 \le i \le N \}$ of morphisms $\phi_i: C_i \to C'_i$ in $\cat$
such that $\phi_{i-1} f_i = f'_i \phi_i$ for $1 \le i \le N$, where $N =
\max \{ n, n' \}$. By convention, $C_i = X$ for all $i>n$ and $C'_i = X'$
for all $i>n'$.

The category $\extn$ has a dual category defined by filtrations. An object
of this category is a couple $(X,F)$, where $X$ is an object of $\cat$ and
$F$ is a filtration of $X$ in $\cat$ of the form
    \[ 0 = F_n \subseteq F_{n-1} \subseteq \dots \subseteq F_0 = X \]
such that $K_i = F_{i-1}/F_i \cong S_{\alpha(i)}$ with $\alpha(i) \in I$
for $1 \le i \le n$. Given an object $(X,F)$ in the dual category, the
corresponding object in $\extn$ is $(X,C)$, where the cofiltration $C$ is
defined by $C_i = X/F_i$ for $0 \le i \le n$, with the natural surjections
$f_i: C_i \to C_{i-1}$. Conversely, if the object $(X,C)$ in $\extn$ is
given, then the corresponding filtration of $X$ is given by $F_i =
\ker(X \to C_i)$ for $0 \le i \le n$, where $X \to C_i$ is the composition
$f_{i+1} \circ \dots \circ f_n: C_n \to C_i$. It is clear from the
construction that the dual objects $(X,C)$ and $(X,F)$ have the same length,
the same factors, and the same order vector.

We recall that a short exact sequence $0 \to Y \to Z \to X \to 0$ in $\cat$
is called an extension of $X$ by $Y$, and that $\ext^1_{\cat}(X,Y)$ denotes
the set of all extensions of $X$ by $Y$, modulo equivalence. The set
$\ext^1_{\cat}(X,Y)$ has a natural $\enm_{\cat}(Y)$-$\enm_{\cat}(X)$
bimodule structure, inherited from the bimodule structure on $\hmm_{\cat}(X,
Y)$.

As the name suggests, the category $\extn$ can be characterized in terms of
extensions. In fact, for any object $(X,C)$ in $\extn$ of length $n$ and for
any integer $i$ with $2 \le i \le n$, the cofiltration $C$ induces a
commutative diagram
\[ \xymatrix{
0 \ar[r] & K_i \ar[r] & C_i \ar[r]^{f_i} & C_{i-1} \ar[r] & 0 \\
0 \ar[r] & K_i \ar[r] \ar[u] & Z \ar[r]_{f_i} \ar[u] & K_{i-1} \ar[r]
\ar[u] & 0
} \]
in $\cat$, where the rows are exact and $Z = f_i^{-1}(K_{i-1})$. We define
$\xi_i \in \ext^1_{\cat}(C_{i-1},K_i)$ and $\tau_i \in \ext^1_{\cat}(K_{i-1},
K_i)$ to be the extensions corresponding to the upper and lower row. By
construction, $\xi_i \mapsto \tau_i$ under the map $\ext^1_{\cat}(C_{i-1},
K_i) \to \ext^1_{\cat}(K_{i-1},K_i)$ induced by the inclusion $K_{i-1}
\subseteq C_{i-1}$. In particular, $C_2$ is an extension of $C_1 = K_1$ by
$K_2$, $C_3$ is an extension of $C_2$ by $K_3$, and in general, $C_{i+1}$ is
an extension of $C_i$ by $K_{i+1}$ for $1 \le i \le n-1$. It follows that $X
= C_n$ is obtained from the factors $\{ K_1, \dots, K_n \} \subseteq \sfam$
by an iterated use of extensions, and this justifies the name \emph{iterated
extensions}.

Let us consider the natural forgetful functor $\extn \to \cat$ given by $(X,
C) \mapsto X$, and the full subcategory $\scat \subseteq \cat$ defined in the
following way: An object $X$ in $\cat$ belongs to $\scat$ if there exists a
cofiltration $C$ of $X$ such that $(X,C)$ is an object of $\extn$. The
following lemma proves that $\scat \subseteq \cat$ is the minimal full
subcategory that contains $\sfam$ and is closed under extensions:

\begin{lemma}
Let $(X',C'), (X'',C'')$ be iterated extensions of the family $\sfam$. If
$X$ is an extension of $X'$ by $X''$ in $\cat$, then there is a cofiltration
$C$ of $X$ such that $(X,C)$ is an iterated extension of the family $\sfam$.
In particular, the full subcategory $\scat \subseteq \cat$ is closed under
extensions.
\end{lemma}
\begin{proof}
Let us assume that $(X',C')$ and $(X'',C'')$ are iterated extensions of the
family $\sfam$ of lengths $n',n''$. Since $X$ is an extension of $X''$ by
$X'$, we can construct a cofiltration of $X$ of length $n = n' + n''$ in the
following way: Let $f: X' \to X$ and $g: X \to X''$ be the maps given by the
extension $0 \to X' \to X \to X'' \to 0$, let $F'$ be the filtration of $X'$
dual to the cofiltration $C'$, and let $F''$ be the filtration of $X''$ dual
to the cofiltration $C''$. We define $F_i = g^{-1}(F''_i)$ for $0 \le i \le
n''$, and $F_i = f(F'_{i-n''})$ for $n'' \le i \le n$. Then $F$ is a
filtration of $X$, and $F_{i-1}/F_i \cong \ker(X \to C''_{i-1})/\ker(X \to
C''_i) \cong K''_i$ for $0 \le i \le n''$, and $F_{i-1}/F_i \cong K'_{i-n''}$
for $n'' \le i \le n$. Let $C$ be the cofiltration of $X$ dual to the
filtration $F$. Then it follows by construction that $(X,C)$ is an iterated
extension of the family $\sfam$ of length $n$.
\end{proof}

We recall that $\scat \subseteq \cat$ is called an \emph{exact Abelian
subcategory} if the inclusion functor $\scat \to \cat$ is an exact functor.
It is well-known that this is the case if and only if $\scat$ is closed in
$\cat$ under kernels, cokernels and finite direct sums. It is clear that
$\scat$ is closed under finite direct sums since it closed under extensions.
But in general, it is not closed under kernels and cokernels.

\begin{proposition} \label{p:abcat}
The full subcategory $\scat \subseteq \cat$ is an exact Abelian subcategory
if and only if the following conditions hold:
\begin{enumerate}
\item $\enm_{\cat}(S_{\alpha})$ is a division algebra for all $\alpha \in I$
\item $\mor_{\cat}(S_{\alpha},S_{\beta}) = 0$ for all $\alpha, \beta \in I$
with $\alpha \ne \beta$
\end{enumerate}
If this is the case, then $\sfam$ is the set of simple objects in $\scat$,
up to isomorphism.
\end{proposition}
\begin{proof}
This follows from Theorem 1.2 in Ringel \cite{ring76}, and the comments
preceding it.
\end{proof}

Let us use the notation from Ringel \cite{ring76}, and say that an object $X$
in $\cat$ is a \emph{point} if $\enm_{\cat}(X)$ is a division ring, and that
two points $X,Y$ in $\cat$ are \emph{orthogonal} if $\mor_{\cat}(X,Y) = 0$
and $\mor_{\cat}(Y,X) = 0$. Moreover, we shall write $k(X) = \enm_{\cat}(X)$
for the division algebra over $k$ associated with a point $X$, and say that
$X$ is a \emph{$k$-rational point} if $k(X) = k$.

\section{Length categories}

A \emph{length category} is an Abelian category such that any of its objects
has finite length, and such that the isomorphism classes of objects form a
set. We recall some well-known facts about length categories; see for
instance Gabriel \cite{gabr73}:
\begin{enumerate}
\item
The Jordan-H\"older Theorem: Any object $X$ in a length category has a
composition series; that is, it has a filtration
    \[ 0 = F_n \subseteq F_{n-1} \subseteq \dots \subseteq F_1 \subseteq F_0
    = X \]
such that $K_i = F_{i-1}/F_i$ is a simple object for $1 \le i \le n$. The
length $n$ and the simple factors $K_1, \dots, K_n$ in a composition series
are unique, up to a permutation of the simple factors.
\item
The Krull-Schmidt Theorem: Any object $X$ in a length category is a finite
direct sum
    \[ X = X_1 \oplus X_2 \oplus \dots \oplus X_r \]
of indecomposable objects. The indecomposable direct summands $X_1, \dots,
X_r$ are unique, up to a permutation.
\item
Mitchell's Embedding Theorem: A length category is exact equivalent to an
exact subcategory of $\mcat A$ for an associative ring $A$.
\end{enumerate}

Let $\sfam$ be a family of orthogonal points in an Abelian $k$-category
$\cat$. It follows from Proposition \ref{p:abcat} that $\scat$ is a length
category, with $\sfam$ as its simple objects. In fact, any length category
which is an Abelian $k$-category is of this type.

Our goal is to classify and explicitly construct the indecomposable objects
in the length category $\scat$. Even though this is a quite hopeless task in
general, we prove a classification result in the special case of uniserial
length categories in Section \ref{s:uniser} and \ref{s:ind}. Our philosophy
is to start with the family $\sfam$, and use iterated extensions in $\extn$
to build larger indecomposable modules.

The \emph{species} of the length category $\scat$ consists of the family
$\{ k(S_{\alpha}): \alpha \in I \}$ of division algebras of its simple
objects, and the family $\{ \ext^1_{\cat}(S_{\alpha},S_{\beta}): \alpha,
\beta \in I \}$ of $k(S_{\beta})$-$k(S_{\alpha})$ bimodules of extensions.
If $\sfam$ is a family of orthogonal $k$-rational points, the species of
$\scat$ can be represented by a quiver $\Lambda$, with $I$ as nodes, and
with $\dim_k \ext^1_{\cat}(S_{\alpha},S_{\beta})$ arrows from node $\alpha$
to node $\beta$ for all $\alpha, \beta \in I$. The quiver $\Lambda$ (and more
generally, the species) of the length category $\scat$ contains a lot of
information about $\scat$ and its indecomposable objects.

In fact, we shall show in Section \ref{s:defm} that the iterated extensions
in $\extn$ are completely determined by noncommutative deformations of its
simple factors. In the unobstructed case, these deformations are determined
by the species of $\scat$.

\section{Uniserial length categories} \label{s:uniser}

Let $\sfam$ be a family of orthogonal $k$-rational points in an Abelian
$k$-category $\cat$, and let $\scat$ be the corresponding length category.
We denote by $\Lambda$ the quiver of the species of $\scat$.

We say that an object $X$ in $\scat$ is \emph{uniserial} if its lattice of
subobjects is a chain. If this is the case, then this chain is the unique
decomposition series of $X$. It follows that $X$ is uniserial if and only
if any two cofiltrations of $X$ are isomorphic. Any uniserial object in
$\scat$ is indecomposable, but the opposite implication does not hold in
general. We say that $\scat$ is a \emph{uniserial category} if every
indecomposable object in $\scat$ is uniserial.

\begin{lemma} \label{l:induni}
Let $X$ be an object in $\scat$, and consider the following conditions:
\begin{enumerate}
\item $X$ is uniserial
\item $X$ has a unique minimal subobject $S \subseteq X$
\item $X$ is indecomposable
\end{enumerate}
Then we have $(1) \Rightarrow (2) \Rightarrow (3)$. In particular, all
conditions are equivalent if and only if $\scat$ is a uniserial category.
\end{lemma}
\begin{proof}
The implication $(1) \Rightarrow (2)$ is obvious. To prove $(2) \Rightarrow
(3)$, let $X = Y_1 \oplus Y_2$ be a direct decomposition of $X$ with $Y_1,
Y_2 \ne 0$. Then there are minimal subobjects $S_i \subseteq Y_i$ in $\scat$
for $i = 1,2$ and this contradicts $(2)$. The last part follows directly
from the definition.
\end{proof}

The implication $(3) \Rightarrow (1)$ in Lemma \ref{l:induni} clearly holds
if $X$ has length $n=2$, since an indecomposable object of length $2$ is a
non-split extension of two objects in $\sfam$. But already for $n=3$, it is
easy to find examples where this implication fails:

\begin{lemma} \label{l:ce1}
If $\sfam$ contains orthogonal $k$-rational points $S,T$ with $\dim_k
\ext^1_{\cat}(S,T) \ge 2$, then there exists an indecomposable, non-uniserial
object in $\scat$ of length $n=3$.
\end{lemma}
\begin{proof}
Notice that there exist non-split extensions $U,V$ of $S$ by $T$ such that
$U$ and $V$ are not isomorphic in $\scat$. In fact, if $U,V$ are non-split
extensions of $S$ by $T$, then any isomorphism $u: U \to V$ satisfies $u(T)
\subseteq T$ since $T$ is the unique minimal subobject of $U,V$ in $\scat$.
Therefore, $u$ induces automorphisms on $T$, and on $S$, which are given by
multiplication in $k^*$ since $S,T$ are $k$-rational points. This means that
not all non-split extensions are isomorphic in $\scat$; otherwise, we would
have that $\dim_k \ext^1_{\cat}(S,T) \le 1$. We define $X = \coker(f)$,
where $f: T \to U \oplus V$ is the diagonal map, and consider the short exact
sequence
    \[ 0 \to T \to U \oplus V \to X \to 0 \]
We see that $X$ has length $n = 3$, that $U,V \subseteq X$ are subobjects in
$\scat$ of length $n = 2$, and that $T$ is the unique minimal subobject of
$X$ in $\scat$. In fact, if $T'$ is another minimal subobject of $X$, then
$T'$ is not contained in $U,V$ since they are uniserial. Hence $U \oplus T'
= X = V \oplus T'$, and this implies that $U$ is isomorphic to $V$, which is
a contradiction. Since $X$ has a unique minimal subobject in $\scat$, it is
indecomposable, and it is non-uniserial since
    \[ 0 \subseteq T \subseteq U \subseteq X \qquad \text{and} \qquad
    0 \subseteq T \subseteq V \subseteq X \]
are different composition series of $X$.
\end{proof}

\begin{lemma} \label{l:ce2}
If $\sfam$ contains points $S,T,U$ such that $\ext^1_{\cat}(U,S),
\ext^1_{\cat}(U,T) \ne 0$ and $S,T$ are orthogonal, then there exists an
indecomposable object in $\scat$ of length $n=3$ with $S,T$ as minimal
subobjects.
\end{lemma}
\begin{proof}
Let $\xi_1, \xi_2 \ne 0$ be non-split extensions of $U$ by $S$, and of $U$ by
$T$, given by short exact sequences
    \[ 0 \to S \to E_1 \xrightarrow{g_1} U \to 0 \qquad \text{and} \qquad
    0 \to T \to E_2 \xrightarrow{g_2} U \to 0 \]
Define $X \subseteq E_1 \oplus E_2$ to be the pullback of $g_1$ and $g_2$.
Then the short exact sequence
    \[ 0 \to S \oplus T \to X \to U \to 0 \]
represents the direct sum extension $(\xi_1,\xi_2) \in \ext^1_{\cat}(U, S
\oplus T)$. It is clear that $S,T$ are minimal subobjects of $X$. Moreover,
$X$ is clearly indecomposable; otherwise, it would have $S$ or $T$ as a
direct summand, and this is not possible since $\xi_1, \xi_2 \neq 0$.
\end{proof}

\begin{lemma} \label{l:ce3}
If $\sfam$ contains points $S,T,U$ such that $\ext^1_{\cat}(S,U),
\ext^1_{\cat}(T,U) \ne 0$ and $S,T$ are orthogonal, then there exists an
indecomposable, non-uniserial object in $\scat$ of length $n=3$.
\end{lemma}
\begin{proof}
Let $\xi_1, \xi_2 \ne 0$ be non-split extensions of $S$ by $U$, and of $T$ by
$U$, given by short exact sequences
    \[ 0 \to U \xrightarrow{f_1} E_1 \to S \to 0 \qquad \text{and} \qquad
    0 \to U \xrightarrow{f_2} E_2 \to T \to 0 \]
Define $X$ be the push-out of $f_1$ and $f_2$. Then the induced short exact
sequence
    \[ 0 \to U \to X \to S \oplus T \to 0 \]
represents the direct sum extension $(\xi_1,\xi_2) \in \ext^1_{\cat}(S \oplus
T,U)$. Clearly, there are natural injections $E_1 \to X$ and $E_2 \to
X$, which gives the composition series
    \[ 0 \subseteq U \subseteq E_1 \subseteq X \qquad \text{and} \qquad
    0 \subseteq U \subseteq E_2 \subseteq X \]
of $X$. Hence, $X$ is not uniserial. Moreover, $U$ is the unique minimal
subobject of $X$ since $\xi_1, \xi_2 \neq 0$, and it follows that $X$ is
indecomposable.
\end{proof}

\begin{proposition} \label{p:uniser-crit}
Let $\sfam = \{ S_{\alpha}: \alpha \in I \}$ be a family of orthogonal
$k$-rational points in an Abelian $k$-category $\cat$, and let $\Lambda$ be
the quiver of the species of the length category $\scat$. If $\scat$ is
uniserial, then following conditions hold:
\begin{enumerate}
\item For any $\alpha \in I$, there is at most one arrow in $\Lambda$ with
source $\alpha$.
\item For any $\beta \in I$, there is at most one arrow in $\Lambda$ with
target $\beta$.
\end{enumerate}
\end{proposition}
\begin{proof}
It follows from Lemma \ref{l:ce1}, Lemma \ref{l:ce2} and Lemma \ref{l:ce3}
that the length category $\scat$ is not uniserial if the quiver $\Lambda$
contains a subquiver of one of the forms
\UseComputerModernTips
\begin{equation*}
\xymatrix{\alpha \ar@<-.5ex>[r] \ar@<.5ex>[r] & \beta }
\qquad \qquad
\xymatrix{\beta & \alpha \ar[l] \ar[r] & \beta'}
\qquad \qquad
\xymatrix{\alpha \ar[r] & \beta & \alpha' \ar[l]}
\end{equation*}
with $\beta \neq \beta'$ in the middle quiver and $\alpha \neq \alpha'$ in
the right quiver.
\end{proof}

The two conditions in Proposition \ref{p:uniser-crit} are equivalent to the
following condition, which we call the \emph{uniseriality criterion} and
refer to as (\uc):
\begin{align*} \tag{\uc}
\sum_{\beta \in I} \, \dim_k \ext^1_{\cat}(S_{\alpha},S_{\beta}) \le 1 & \quad
\text{for all $\alpha \in I$} \\
\sum_{\alpha \in I} \, \dim_k \ext^1_{\cat}(S_{\alpha},S_{\beta}) \le 1 & \quad
\text{for all $\beta \in I$}
\end{align*}
We claim that the length category $\scat$ is uniserial if and only if the
condition (\uc) holds. It turns out that this was known already in the
60's; see Section 8.3 in Gabriel \cite{gabr73}. As far as we know, this
result first appeared in Amdal, Ringdal \cite{amda-ring68}, where it is
stated without proof.

We shall give an elementary and constructive proof of this characterization
in the next section. The proof is constructive in the sense that we classify
and explicitly construct all indecomposable objects in $\scat$ when (\uc)
holds, and show that these indecomposable objects are uniserial.

\section{Construction of indecomposable objects} \label{s:ind}

Let $\sfam$ be a family of orthogonal $k$-rational points in an Abelian
$k$-category $\cat$, and let $\scat$ be the corresponding length category. We
denote by $\Lambda$ the quiver of the species of $\scat$, and assume that the
condition (\uc) holds. In this situation, we shall classify and explicitly
construct all indecomposable objects in $\scat$.

We consider the full subcategory $\extncl{n,*} \subseteq \extn$ of iterated
extensions $(X,C)$ of length $n$ such that $\xi_2, \dots, \xi_n \neq 0$. Any
indecomposable object $X$ in $\scat$ of length $n$ has a cofiltration $C$ such
that $(X,C)$ is an iterated extension in $\extn$ with $\xi_n \neq 0$. The
idea is that many indecomposable objects, though not necessarily all, have a
cofiltration $C$ such that $(X,C)$ is in $\extncl{n,*}$, and we start by
classifying these indecomposable objects.

\begin{lemma} \label{l:indextnstr}
Let $(X,C)$ be an iterated extension in $\extncl{n,*}$. If $\sfam$ satisfies
(\uc), then the $k$-linear map
    \[ \ext^1_{\cat}(C_{i-1},K) \to \ext^1_{\cat}(K_{i-1},K) \]
induced by the inclusion $K_{i-1} \subseteq C_{i-1}$ is an isomorphism for
all integers $i$ such that $2 \le i \le n$ and for all simple objects $K \in
\sfam$. In particular, $\tau_i \neq 0$ for $2 \le i \le n$.
\end{lemma}
\begin{proof}
We show the result by induction on $n$. Since $C_1=K_1$ by definition, the
result is clearly true for $n=2$. So let $n \ge 3$, and assume that the
result holds for all integers less than $n$ and all simple objects $K \in
\sfam$. In particular, this implies that
    \[ \ext^1_{\cat}(C_{n-2},K) \to \ext^1_{\cat}(K_{n-2},K) \]
is an isomorphism. Since $\xi_{n-1} \mapsto \tau_{n-1}$ under this map when
$K = K_{n-1}$, it follows that $\tau_{n-1} \neq 0$, and in particular, that
$\ext^1_{\cat}(K_{n-2},K_{n-1}) = k \cdot \tau_{n-1} \cong k$ by (\uc).
Hence we also have $\ext^1_{\cat}(C_{n-2},K) = k \cdot \xi_{n-1} \cong k$.
Let us consider the long exact sequence of the functor $\hmm_{\cat}(-,K)$
applied to the extension $\xi_{n-1}$, given by
\begin{align*}
\dots \to \hmm_{\cat}(C_{n-1},K) &\to \hmm_{\cat}(K_{n-1},K) \to \\
\ext^1_{\cat}(C_{n-2},K) \to \ext^1_{\cat}(C_{n-1},K) &\to
\ext^1_{\cat}(K_{n-1},K) \to \dots
\end{align*}
For all simple objets $K \in \sfam$, we claim that $\hmm_{\cat}(K_{n-1},K)
\cong \ext^1_{\cat}(K_{n-2},K)$ and that $\hmm_{\cat}(K_{n-1},K) \to
\ext^1_{\cat}(C_{n-2},K)$ is an isomorphism: If $K = K_{n-1}$, we have that
$\enm_{\cat}(K_{n-1}) \cong k$ and that $\ext^1_{\cat}(K_{n-2},K_{n-1}) = k
\cdot \tau_{n-1} \cong k$. This proves the claim, since
$\ext^1_{\cat}(C_{n-2},K_{n-1}) = k \cdot \xi_{n-1} \cong k$ by the comments
above and the identity on $K_{n-1}$ maps to $\xi_{n-1}$ by construction. If
$K$ is not isomorphic to $K_{n-1}$, we have that $\hmm_{\cat}(K_{n-1},K) = 0$
and that $\ext^1_{\cat}(K_{n-2},K) = 0$ by orthogonality and (\uc). This
proves the claim, since $\ext^1_{\cat}(C_{n-2},K) \cong
\ext^1_{\cat}(K_{n-2},K) = 0$ by the induction hypothesis. We conclude that
for all simple objects $K \in \sfam$, the $k$-linear map
    \[ \ext^1_{\cat}(C_{n-1},K) \to \ext^1_{\cat}(K_{n-1},K) \]
is injective, and it is enough to show that it is an isomorphism to conclude
the proof. If $K = K_n$, then $\xi_n \mapsto \tau_n$ under this map, and by
injectivity, it follows that $\tau_n \neq 0$ since $\xi_n \neq 0$. Therefore,
$\ext^1_{\cat}(K_{n-1},K_n) = k \cdot \tau_n \cong k$ by (\uc), and the map
is an isomorphism. If $K$ is not isomorphic to $K_n$, then it follows from
(\uc) that $\ext^1_{\cat}(K_{n-1},K) = 0$, and the map is an isomorphism also
in this case.
\end{proof}

We consider the map $v_n: \extncl{n,*} \to I^n$, which maps an iterated
extension $(X,C)$ to its order vector $\vect \alpha$, and say that a vector
$\vect \alpha \in I^n$ is \emph{admissible} if $\vect \alpha \in \im(v_n)$.
For any iterated extension $(X,C)$ in $\extncl{n,*}$, it follows from Lemma
\ref{l:indextnstr} that $\tau_i \neq 0$ for $2 \le i \le n$, and that
    \[ \ext^1_{\cat}(K_{i-1},K_i) = k \cdot \tau_i \cong k \quad \text{and} \quad
    \ext^1_{\cat}(C_{i-1},K_i) = k \cdot \xi_i \cong k \]
This means that if $\vect \alpha$ is admissible, and $(X,C), (X',C')$ are
two iterated extensions in $\extncl{n,c}$ with order vector $\vect \alpha$,
then $X \cong X'$ in $\scat$. In fact, we have that $\xi'_i = c_i \xi_i$ for
$2 \le i \le n$ with $c_i \in k^*$, and it is well-known that the extensions
of $C_{i-1}$ by $K_i$ in the same $k^*$-orbit of $\ext^1_{\cat}(C_{i-1},K_i)$
are isomorphic in $\scat$.

Let $\vect \alpha \in I^n$ be an admissible vector. Then there is an iterated
extension $(X,C)$ in $\extncl{n,*}$ with order vector $\vect \alpha$, and it
follows from the comments above that $X$ is unique, up to isomorphism in
$\scat$. We shall write $X(\vect \alpha)$ for this object in $\scat$ when
$\vect \alpha$ is admissible. It turns out that $X(\vect \alpha)$ is an
indecomposable and uniserial object in $\scat$:

\begin{proposition} \label{p:ind}
Let $\vect \alpha \in I^n$ be an admissible vector. If $\sfam$ satisfies
(\uc), then $X(\vect \alpha)$ is indecomposable and uniserial.
\end{proposition}
\begin{proof}
We claim that if $(X,C)$ is an iterated extension in $\extncl{n,*}$, then
there is a unique minimal subobject of $X$. The claim clearly holds if $n =
2$, and we shall prove the claim by induction on $n$. We therefore assume
that $n \ge 3$, and that the claim holds for all iterated extensions of
length less than $n$. To prove that it holds for iterated extensions of
length $n$, it is enough to prove to $\phi(K) = K_n$ for any injective
homomorphism $\phi: K \to X$. We consider the commutative diagram
\[ \xymatrix{
0 \ar[r] & K_n \ar[r] \ar[d] & f_n^{-1}(K_{n-1}) \ar[r] \ar[d] & K_{n-1}
\ar[r] \ar[d] & 0 \\
0 \ar[r] & K_n \ar[r] & X \ar[r]_{f_n} & C_{n-1} \ar[r] & 0
} \]
given by the cofiltration $C$, where the horizontal rows represent $\tau_n$
and $\xi_n$. Assume that $\phi(K) \neq K_n$, which implies that $\phi(K)
\cap K_n = 0$ since $K_n$ is simple, and consider the induced morphism $f_n
\circ \phi: K \to C_{n-1}$. We claim that this morphism is injective. In
fact, we have that $\ker(f_n \circ \phi) = \phi^{-1}(K_n) = 0$. By the
induction hypothesis, this means that $f_n \circ \phi(K) = K_{n-1}$. This
implies that $\phi(K) \subseteq f_n^{-1}(K_{n-1})$, which is a contradiction
since $\tau_n \neq 0$ by Lemma \ref{l:indextnstr}, and it follows that
$\phi(K) = K_n$. We have therefore proven the induction step, which means
that $X$ is indecomposable with a unique minimal subobject in $\scat$.
Finally, $X$ is uniserial since $C_i$ has a unique minimal submodule in
$\scat$ for $2 \le i \le n$. In fact, we can see this by applying the
argument above to the iterated extension $(C_i,C')$ in $\extncl{i,*}$, where
$C'$ is the cofiltration
    \[ C_i \xrightarrow{f_i} C_{i-1} \to \dots \to C_2 \xrightarrow{f_2}
    C_1 \to C_0 = 0 \]
obtained by capping $C$ at $C_i$.
\end{proof}

The next step in the classification, is to characterize the vectors $\vect
\alpha \in I^n$ that are admissible. If $\vect \alpha$ is admissible, then
by definition there is an iterated extension $(X,C)$ in $\extncl{n,*}$ with
order vector $\vect \alpha$, and it follows from Lemma \ref{l:indextnstr}
that
    \[ \ext^1_{\cat}(K_{i-1},K_i) = k \cdot \tau_i \neq 0 \]
for $2 \le i \le n$. This means that $\vect \alpha$ corresponds to a path of
length $n-1$ in the quiver $\Lambda$, with an arrow from node $\alpha(i-1)$
to node $\alpha(i)$ for $2 \le i \le n$.

Conversely, if $\vect \alpha \in I^n$ is a vector corresponding to a path of
length $n-1$ in the quiver $\Lambda$, such that
    \[ \ext^1_{\cat}(K_{i-1},K_i) \neq 0 \]
for $2 \le i \le n$, with $K_i = S_{\alpha(i)}$, then we may choose $\sigma_i
\in \ext^1_{\cat}(K_{i-1},K_i)$ with $\sigma_i \neq 0$. The vector $\vect
\alpha$ is admissible if there is an iterated extensions $(X,C)$ of length
$n$ in $\extn$ with $\tau_i = \sigma_i$ for $2 \le i \le n$, where $\tau_i$
is the extension induced by the cofiltration $C$. This is clearly the case
when $n = 2$, since $\sigma_2 \neq 0$ is a non-split extension
    \[ 0 \to K_2 \to E \to K_1 \to 0 \]
of $K_1$ by $K_2$. In fact, we may choose $X = E$ and the cofiltration $C$ of
length $n = 2$ given by $E \to K_1 \to 0$, which has $\tau_2 = \sigma_2$.
However, if $n \ge 3$, there are obstructions for the existence of such an
iterated extension:

\begin{proposition} \label{p:obstr-mmp}
Let $\vect \alpha \in I^n$ be a vector corresponding to a path of length
$n-1$ in the quiver $\Lambda$, and choose a non-zero extension $\sigma_i \in
\ext^1_{\cat}(K_{i-1},K_i)$ for $2 \le i \le n$. If $\sfam$ satisfies (\uc),
then $\vect \alpha$ is admissible if and only if the matric Massey product
    \[ \langle \sigma_2, \sigma_3, \dots, \sigma_n \rangle \]
is defined and contains zero.
\end{proposition}
\begin{proof}
Since $\scat$ is exact equivalent to an exact subcategory of a category of
modules over an associative $k$-algebra, we may assume that $\cat$ is such a
module category without loss of generality. In this case, the result follows
from Proposition 4 in Eriksen, Siqveland \cite{erik-siqv17}, and the
preceding construction.
\end{proof}

We remark that the use of an exact embedding of $\cat$ into a module category
is a choice of convenience in the proof of Proposition \ref{p:obstr-mmp}, and
not essential. Matric Massey products are tied to noncommutative deformations
and may be computed directly in many Abelian $k$-categories; see Eriksen,
Laudal, Siqveland \cite{erik-laud-siqv17}.

We say that $\scat$ is a \emph{hereditary length category} if $\ext^2_{\cat}
(S,T) = 0$ for any objects $S,T \in \sfam$. If this is the case, then the
obstruction in Proposition \ref{p:obstr-mmp} vanishes. This is clear from the
construction of matric Massey products.

We claim that any indecomposable object $X$ in $\scat$ has the form $X(\vect
\alpha)$ for an admissible vectors $\vect \alpha \in I^n$, and this would
complete the classification. To prove the claim, we must show that any
indecomposable object $X$ in $\scat$ has a cofiltration $C$ such that $(X,C)
\in \extncl{n,*}$:

\begin{proposition} \label{p:ind-nz}
Let $(X,C)$ be an iterated extension of length $n$ in $\extn$ such that $X$
is indecomposable. If $\sfam$ satisfies (\uc), then $\xi_i \ne 0$ for $2 \le
i \le n$.
\end{proposition}
\begin{proof}
The result clearly holds if $n = 2$, and we shall prove the claim by
induction on $n$. We therefore assume that $n \ge 3$, and that the claim
holds for all iterated extensions of length less than $n$. For an iterated
extension $(X,C)$ of length $n$ in $\extn$, where $X$ is indecomposable,
there is a non-split, short exact sequence $0 \to K_n \to X \to C_{n-1} \to
0$, and $C_{n-1}$ has a direct decomposition
    \[ C_{n-1} = Y_1 \oplus \dots \oplus Y_q \]
such that $Y_j$ is an indecomposable object in $\scat$ of length $n_j$ for $1
\le j \le q$. If $q = 1$, then $\xi_i \neq 0$ for $2 \le i \le n-1$ by the
induction hypothesis, and $\xi_n \neq 0$ since $X$ in indecomposable. Hence
we have proved the induction step if $q = 1$. Next, we suppose that $q > 1$,
and show that this leads to a contradiction: Choose a cofiltration of $Y_j$
given by
    \[ Y_j = C_{j,n_j} \xrightarrow{f_{j,n_j}} C_{j,n_j-1} \to \dots \to
    C_{j,1} \xrightarrow{f_{j1}} C_{j,0} = 0 \]
for $1 \le j \le q$, with $K_{ji} = \ker(f_{ji}) \cong S_{\alpha(j,i)}$ for
$1 \le i \le n_j$ and for some $\alpha(j,i) \in I$. Since
    \[ \ext^1_{\cat}(C_{n-1},K_n) \cong \osum_j \ext^1_{\cat}(Y_j, K_n) \]
we have $\xi_n = (\xi_{n,1}, \dots, \xi_{n,q})$, and we claim that $\xi_{n,j}
\ne 0$ for $1 \le j \le q$. In fact, if $\xi_{n,j}=0$ for some $j$, then the
short exact sequence
    \[ 0 \to K_n \to f_n^{-1}(Y_j) \to Y_j \to 0 \]
splits, hence there is a section of $X \to Y_j$ making $Y_j$ a direct summand
of $X$. This is a contradiction, and therefore we must have $\xi_{n,j} \neq
0$ for $1 \le j \le q$. Since $Y_j$ is indecomposable of length $n_j<n$ for
$1 \le j \le q$, it follows from the induction hypothesis that the extensions
$\xi_{ij} \in \ext^1_{\cat}(C_{j,i-1},K_{ji}) \ne 0$ for $2 \le i \le n_j$.
Let $X_j = f_n^{-1}(Y_j) \subseteq X$, with induced surjective morphism $f_n:
X_j \to Y_j = C_{j,n_j}$. Then $(X_j,C_{j\ast})$ is an iterated extension in
$\extncl{n_j,*}$, given by
    \[ X_j = C_{j,n_j+1} \xrightarrow{f_n} N_j = C_{j,n_j} \to \cdots \to
    C_{j,1} \to C_{j,0} = 0 \]
Therefore, it follows from Lemma \ref{l:indextnstr} that $\ext^1_{\cat}(N_j,
K_n) \to \ext^1_{\cat}(K_{j,n_j},K_n)$ is an isomorphism and that the image
$\tau_{n,j}$ of $\xi_{n,j}$ is non-zero for $1 \le j \le n$. Hence, we must
have that
    \[ \alpha(1,n_1) = \alpha(2,n_2) = \dots = \alpha(q,n_q) \]
is the unique node in $\Lambda$ with an arrow to node $\alpha(n)$. In a
similar manner, it follows that $\ext^1_{\cat}(C_{j,n_j-i},K_{j,n_j-i+1}) \to
\ext^1_{\cat}(K_{j,n_j-i},K_{j,n_j-i+1})$ is an isomorphism and therefore that
$\tau_{j,n_j-i+1} \neq 0$ for $1 \le j \le q$ and $1 \le i \le \min(n_1, \dots,
n_q) -1$. Hence we must have that
    \[ \alpha(1,n_1-i) = \alpha(2,n_2-i) = \dots = \alpha(q,n_q-i) \]
Without loss of generality, we may assume that $n_1 \le n_2 \le \dots \le
n_q$. From the argument above, it follows that $Y_1 \subseteq Y_2 \subseteq
\dots \subseteq Y_q$. For any injective morphism $Y_1 \to Y_1 \oplus Y_2$ that
has the form $y_1 \mapsto (y_1, C y_1)$ with $C \in k$, there is an induced
injective map $i: Y_1 \to C_{n-1} = N_1 \oplus \dots \oplus N_q$, and we may
consider the commutative diagram
\[ \xymatrix{
0 \ar[r] & K_n \ar[r] \ar@{=}[d] & X \ar[r]^{f_n} & C_{n-1} \ar[r] & 0 \\
0 \ar[r] & K_n \ar[r] & f_n^{-1}(i(Y_1)) \ar[r] \ar[u] & Y_1 \ar[r]
\ar[u]_i & 0
} \]
where the first row is the extension $\xi_n = (\xi_{n,1}, \dots, \xi_{n,q})$.
We claim that we may choose $C$ such that second row is a split extension. In
fact, this follows from the fact that there are $c, \; c_i \in k^*$ such that
$\tau_{n,1} = c \cdot \tau_{n,2}$ and
    \[ \tau_{1,n_1-i+1} = c_i \cdot \tau_{2,n_2-i+1} \]
for $1 \le i \le n_1-1$. It follows that $Y_1$ is a direct summand in $X$,
and this contradicts the assumption $q > 1$.
\end{proof}

\begin{theorem} \label{t:exclass}
Let $\sfam = \{ S_{\alpha}: \alpha \in I \}$ be a family of orthogonal
$k$-rational points in an Abelian $k$-category $\cat$. If $\sfam$ satisfies
(\uc), then the indecomposable objects in $\scat$ of length $n$ are given by
$\{ X(\vect \alpha): \vect \alpha \in \mathcal J \}$, up to isomorphism in
$\scat$, where the subset $\mathcal J \subseteq I^n$ consists of the vectors
$\vect \alpha$ such that the following conditions hold:
\begin{enumerate}
\item
$\ext^1_{\cat}(S_{\alpha(i-1)},S_{\alpha(i)}) \neq 0$ for $2 \le i \le n$
\item
If $\sigma_i \in \ext^1_{\cat}(S_{\alpha(i-1)},S_{\alpha(i)})$ is non-zero
for $2 \le i \le n$, then the matric Massey product $\langle \sigma_2,
\sigma_3, \dots, \sigma_n \rangle$ is defined and contains zero.
\end{enumerate}
Moreover, the indecomposable objects $X(\vect \alpha)$ are uniserial, and
can be constructed from the family $\sfam$ and their extensions.
\end{theorem}
\begin{proof}
This follows from the results in this section. The explicit construction
of $X(\vect \alpha)$ is obtained by iteratively constructing $C_i$ as an
extension of $C_{i-1}$ by $K_i$ for $2 \le i \le n$.
\end{proof}

\begin{corollary} \label{c:uniser-crit}
Let $\sfam = \{ S_{\alpha}: \alpha \in I \}$ be a family of orthogonal
$k$-rational points in an Abelian $k$-category $\cat$. Then $\scat$ is
uniserial if and only if (\uc) holds.
\end{corollary}
\begin{proof}
This follows from Proposition \ref{p:uniser-crit} and Theorem
\ref{t:exclass}.
\end{proof}

There is a more general form of Corollary \ref{c:uniser-crit}, where the
points in $\sfam$ are not assumed to be $k$-rational; see Gabriel
\cite{gabr73}. We have chosen to work with $k$-rational points out of
convenience, and also because all points are $k$-rational in the
applications we have in mind. However, it would be possible to prove the
general form of Theorem \ref{t:exclass} and Corollary \ref{c:uniser-crit}
using the methods of this paper.

\section{Graded holonomic D-modules on monomial curves}

Let $\Gamma \subseteq \n_0$ be a numerical semigroup, generated by positive
integers $a_1, \dots, a_r$ without common factors, and let $A = k[\Gamma]
\cong k[t^{a_1}, \dots, t^{a_r}]$ be its semigroup algebra over the field $k
= \cc$ of complex numbers. We call $A$ a \emph{monomial curve} since $X =
\spec(A)$ is the affine monomial curve $X = \{ (t^{a_1}, t^{a_2}, \dots,
t^{a_r}): t \in k \} \subseteq \aff^r_k$.

We studied the positively graded algebra $D$ of differential operators on the
monomial curve $A = k[\Gamma]$ in Eriksen \cite{erik03}, and the category
$\grdhol$ of graded holonomic left $D$-modules in Eriksen \cite{erik18}. We
recall that any $D$-module $M$ satisfies the Bernstein inequality $d(M) \ge
1$, that $M$ is holonomic if $d(M) = 1$, and that this condition holds if and
only if $M$ has finite length; see Proposition 4 and Proposition 5 in Eriksen
\cite{erik18}. This implies that $\grdhol$ is a length category, and its
simple objects are given by
    \[ \{ M_0[w]: w \in \z \} \cup \{ M_{\alpha}[w]: \alpha \in J^*, \; w \in
    \z \} \cup \{ M_{\infty}[w]: w \in \z \} \]
where $J^* = \{ \alpha \in \cc: 0 \le \re(\alpha) < 1, \; \alpha \neq 0 \}$;
see Theorem 10 in Eriksen \cite{erik18}. Moreover, the graded extensions of
the simple objects are given by
    \[ \ext^1_D(M_{\alpha}[w],M_{\beta}[w'])_0 = \begin{cases} k \xi \cong k,
    & (\alpha,\beta) = (0,\infty), (\infty,0) \text{ and } w = w'\\ k \xi
    \cong k, & \alpha = \beta \in J^* \text{ and } w = w' \\ 0, &
    \text{otherwise} \end{cases} \]
for simple graded $D$-modules $M_{\alpha}[w], M_{\beta}[w']$ with $\alpha,
\beta \in J^* \cup \{ 0, \infty \}$ and $w,w' \in \z$; see Proposition 12 in
Eriksen \cite{erik18}.

\begin{proposition}
The family $\sfam = \{ M_{\alpha}[w]: \alpha \in J^* \cup \{ 0, \infty \}, \;
w \in \z \}$ is the family of simple objects in $\grdhol$, and it is a family
of orthogonal $k$-rational points that satisfies (\uc). In particular, the
category $\grdhol$ of graded holonomic $D$-modules is a uniserial category.
\end{proposition}
\begin{proof}
Since $k = \cc$ is algebraically closed, it follows from the main theorem in
Quillen \cite{quil69} that $\enm_D(M_{\alpha}[w]) = k$ for all $\alpha \in
J^* \cup \{ 0, \infty \}, \; w \in \z$. Moreover, the comments above show
that $\sfam$ is the family of simple objects in $\grdhol$, and therefore a
family or orthogonal $k$-rational points, which satisfies (\uc).
\end{proof}

It is, in principle, possible to construct all indecomposable objects in
$\grdhol$ using the constructive proof of Theorem \ref{t:exclass}. As an
illustration, we shall classify the indecomposable objects in the case $A =
k[t]$, which is the unique smooth monomial curve. The classification would be
similar in the other cases, since all rings of differential operators on
monomial curves are Morita equivalent. However, the indecomposable objects
would be defined by more complicated equations in the singular cases.

Note that when $A = k[t]$, the ring $D$ of differential operators on $A$ is
the first Weyl algebra $\weyl = k[t]\langle\partial\rangle$, with generators
$t$ and $\partial = d/dt$, and relation $[\partial,t] = 1$. Let us write $E =
t \partial$ for the Euler derivation in $D$. The simple objects in $\grdhol$,
up to graded isomorphisms and twists, are given by $M_0 = D/D \partial$,
$M_{\infty} = D/Dt$ and $M_{\alpha} = D/D (E - \alpha)$ for $\alpha \in J^*$.

\begin{theorem}
Let $D = \weyl$ be the first Weyl algebra. The indecomposable graded holonomic
$D$-module, up to graded isomorphisms and twists, are given by
    \[ M(\alpha,n) = D/D \, (E-\alpha)^n, \quad M(\beta,n) = D/D \, \mathbf
    w(\beta,n) \]
where $n \ge 1$, $\alpha \in J^*$, $\beta \in \{ 0, \infty \}$, and $\mathbf
w(\beta,n)$ is the alternating word on $n$ letters in $t$ and $\partial$,
ending with $\partial$ if $\beta = 0$, and in $t$ if $\beta = \infty$.
\end{theorem}
\begin{proof}
Let us write $I = J^* \cup \{ 0, \infty \}$, such that $\sfam = \{ M_{\alpha}
[w]: (\alpha,w) \in I \times \z \}$ is the family of simple objects in
$\grdhol$. It follows from the computation of the graded extensios above that
for any length $n \ge 1$ and any $(\alpha,w) \in I \times \z$, there is a
unique path
    \[ (\alpha,w) = (\alpha(1),w_1) \to (\alpha(2),w_2) \to \dots \to
    (\alpha(n),w_n) \]
in $\Lambda$ such that $\ext^1_D(M(\alpha_{i-1})[w_{i-1}],M(\alpha_i)[w_i])_0
\neq 0$ for $2 \le i \le n$. The corresponding vector is admissible since $D
= \weyl$ is a hereditary graded ring; see for instance Coutinho \cite{cout95}.
Note that if $\alpha \in J^*$, then $\alpha(i) = \alpha$ and $w_i = w$ for $1
\le i \le n$, and if $\alpha \in \{ 0, \infty \}$, then we have
    \[ (\alpha(i),w(i)) = \begin{cases} (\alpha,w), & i \text{ is odd} \\
    (0,w-1), & i \text{ is even}, \; \alpha = \infty \\ (\infty,w+1), &
    i \text{ is even}, \; \alpha = 0 \end{cases} \]
for $1 \le i \le n$. The rest follows from the fact that $\ext^1(C_{i-1},K_i)
\to \ext^1(K_{i-1},K_i)$ is an isomorphism by Lemma \ref{l:indextnstr}, and
that the graded extensions obtained using factorization in $D$, such as
    \[ 0 \to D/D(E-\alpha) \to D/(E-\alpha)^n \to D/D(E-\alpha)^{n-1} \to 0 \]
for $\alpha \in J^*$ and $n \ge 2$, are non-split.
\end{proof}

\section{Iterated extensions and noncommutative deformations} \label{s:defm}

Let $\sfam = \{ S_{\alpha}: \alpha \in I \}$ be a family of orthogonal points
in an Abelian $k$-category $\cat$, and let $\scat$ be the corresponding
length category. In this section, we consider the noncommutative deformations
of finite subfamilies of $\sfam$; see Laudal \cite{laud02} and also Eriksen,
Laudal, Siqveland \cite{erik-laud-siqv17}, and show that they determine the
iterated extensions in $\extn$. This will shed light on some of the results
for uniserial length categories in this paper, and in particular Propostion
\ref{p:obstr-mmp}. It will also provide a useful tool for future study of
length categories that are more complicated than the uniserial ones.

Let $(X,C)$ be an iterated extension in $\extn$ of length $n$ with order
vector $\vect \alpha$. We define the \emph{extension type} of $(X,C)$ to be
the ordered quiver $\Gamma$ with nodes $\{ \alpha(1), \alpha(2), \dots,
\alpha(n) \}$ and edges $\gamma_{i-1,i}$ from node $\alpha(i-1)$ to node
$\alpha(i)$ for $2 \le i \le n$. The quiver is ordered in the sense that
there is a total order $\gamma_{12} < \gamma_{23} < \dots < \gamma_{n-1,n}$
on the edges in $\Gamma$. Clearly, the extension type $\Gamma$ is uniquely
defined by the order vector $\vect \alpha$, and isomorphic iterated
extensions have the same extension type. We denote by $\extniso \Gamma$ the
set of isomorphism classes of iterated extensions of the family $\sfam$ with
extension type $\Gamma$.

To fix notation, we shall give the set $J = \{ \alpha(1), \alpha(2), \dots,
\alpha(n) \} \subseteq I$ of nodes in $\Gamma$ a total order, and write
    \[ S(\Gamma) = \{ S_{\alpha}: \alpha \in J \} = \{ X_1, \dots, X_r \} \]
for the associated subfamily of $\sfam$, considered as an ordered set with
$X_1 < X_2 < \dots < X_r$. Note that for $1 \le l \le r$, we have that $X_l
= S_{\alpha(i)}$ for at least one value of $i$ with $1 \le i \le n$, and
that $r \le n$, with $r < n$ if there are repeated factors.

The \emph{path algebra} $k[\Gamma]$ of the ordered quiver $\Gamma$ is the
$k$-algebra with base consisting of paths $\gamma_{i-1,i} \cdot \gamma_{i,
i+1} \cdot \dots \cdot \gamma_{j-1,j}$ of length $j-i+1$ for $2 \le i \le
j \le n$. The product of two paths $\gamma \cdot \gamma'$ is given by
juxtaposition when the last arrow $\gamma_{j-1,j}$ in the first path
$\gamma$ is the predecessor of the first arrow $\gamma_{j,j+1}$ in the
second path $\gamma'$ in the total ordering, and otherwise the product
$\gamma \cdot \gamma' = 0$. We consider $e_i$ as a path of length $0$ for
$1 \le i \le r$. For example, an iterated extension of length three with
$\alpha(1) < \alpha(2) < \alpha(3)$ has $r = n = 3$, and its extension type
$\Gamma$ has path algebra
    \[ k[\Gamma] = \begin{pmatrix} k & k & k \\ 0 & k & k \\ 0 & 0 & k
    \end{pmatrix} = \begin{pmatrix} k \cdot e_1 & k \cdot \gamma_{12} & k
    \cdot \gamma_{12} \gamma_{13} \\ 0 & k \cdot e_2 & k \cdot \gamma_{23}
    \\ 0 & 0 & k \cdot e_3 \end{pmatrix} \]
We recall that the category $\acat r$ of Artinian $r$-pointed algebras
consists of Artinian $k$-algebras $R$ with $r$ simple modules fitting into
a diagram $k^r \to R \to k^r$, where the composition is the identity. It is
clear that if $\Gamma$ is the extension type of an iterated extension of $r$
objects in $\sfam$ of length $n$, then the path algebra $k[\Gamma]$ is an
algebra in $\acat r(n)$, where $\acat r(n)$ is the full subcategory of $\acat
r$ consisting of algebras $R$ such that $I(R)^n = 0$, with $I(R) = \ker(R \to
k^r)$. Noncommutative deformation functors are defined on the category $\acat
r$, and noncommutative deformations are parameterized by $r$-pointed Artinian
algebras; see Chapter 3 of Eriksen, Laudal, Siqveland \cite{erik-laud-siqv17}
for details.

Let $\Gamma$ be an extension type, and write $\sfam(\Gamma) = \{ X_1, \dots,
X_r \}$ for the corresponding ordered subfamily of $\sfam$. We consider the
noncommutative deformation functor
    \[ \defm{\sfam(\Gamma)}: \acat r \to \sets \]
of the finite family $\sfam(\Gamma)$ in the Abelian category $\cat$. We shall
assume, without loss of generality, that $\cat$ is the category of right
modules over an associative $k$-algebra $A$ in the rest of this section. This
is a choice of convenience, as noncommutative deformations can be computed
directly in many other Abelian $k$-categories; see Eriksen, Laudal, Siqveland
\cite{erik-laud-siqv17}.

\begin{proposition} \label{p:defm-ncdefm-itext}
There is a bijective correspondence between the noncommutative deformations
in $\defm{\sfam(\Gamma)}(k[\Gamma])$ and the set $\mathcal E(\sfam, \Gamma)$
of equivalence classes of iterated extensions of the family $\sfam$ with
extension type $\Gamma$.
\end{proposition}
\begin{proof}
Let us write $\sfam(\Gamma) = \{ X_1, \dots, X_r \}$, and let $\vect \alpha$
be the order vector corresponding to the extension type $\Gamma$. There is a
unique $s$ with $1 \le s \le r$ such that $S_{\alpha(1)} = X_s$. Any
noncommutative deformation $X_{\Gamma} \in \defm{\sfam(\Gamma)}(k[\Gamma])$
has the form $X_{\Gamma} = ( k[\Gamma]_{ij} \otimes_k X_j )$ as a left
$k[\Gamma]$-module by flatness, with a right multiplication of $A$. Let
$X_{\Gamma}(s) = e_{s} \cdot X_{\Gamma} \subseteq X_{\Gamma}$, which is
closed under right multiplication with $A$. A path in $e_{s} \cdot k[\Gamma]$
is called \emph{leading} if it has the form $\gamma_{12} \gamma_{23} \cdot
\gamma_{i-1,i}$ and \emph{non-leading} otherwise. By convention, we consider
the path $e_{s}$ as leading, and define
    \[ X_{\Gamma}^{NL}(s) = \osum_{\gamma} \; \gamma \cdot X_{\Gamma}(s) \]
where the sum is taken over all non-leading paths $\gamma$ in $e_{s} \cdot
k[\Gamma]$. Notice that $X_{\Gamma}^{NL}(s) \subseteq X_{\Gamma}(s)$ is
closed under right multiplication by $A$. We define $X = X_{\Gamma}(s) /
X_{\Gamma}^{NL}(s)$, which has an induced right $A$-module structure. As a 
$k$-linear space, we have that 
\begin{equation*}
X \cong \osum_{1 \le i \le n} \; \left( \gamma_{12} \, \gamma_{23} \dots
\gamma_{i-1,i} \right) \otimes_k S_{\alpha(i)}
\end{equation*}
with $S_{\alpha(i)} = X_l$ for some $l$ with $1 \le l \le r$, and we claim 
that there is a cofiltration $C$ of $X$ such that $(X,C)$ is an iterated 
extension of $\sfam$ with extension type $\Gamma$. In fact, we may choose 
the cofiltration $C$ dual to the filtration $F$ given by
    \[ F_j = \osum_{j+1 \le i \le n} \; \left( \gamma_{12} \, \gamma_{23}
    \dots \gamma_{i-1,i} \right) \otimes_k S_{\alpha(i)} \]
for $0 \le j \le n$, where $F_j \subseteq X$ is closed under right
multiplication with $A$. Conversely, if $(X,C)$ is a iterated extension of
$\sfam$ with extension type $\Gamma$, then it follows from the construction
in Section 3 of Eriksen, Siqveland \cite{erik-siqv17} that
    \[ X \cong K_n \oplus K_{n-1} \oplus \dots \oplus K_2 \oplus K_1 \]
as a $k$-linear vector space, with $K_i \cong S_{\alpha(i)}$ and with right
multiplication of $A$ given by
\begin{equation*}
(m_n, \dots, m_2, m_1)a = (m_n \cdot a + \sum_{i = 1}^{n-1} \psi^{in}_a(m_i),
\dots, m_2 \cdot a + \psi^{12}_a(m_1), m_1 \cdot a)
\end{equation*}
for $m_i \in K_i, \; a \in A$. Let $I_l = \{ i: \alpha(i) = l \}$ for $1 \le
l \le r$. Then the right multiplication of $A$ on $X_{\Gamma} = (
k[\Gamma]_{ij} \otimes_k X_j )$ given by
    \[ (e_l \otimes m_l) \cdot a = e_l \otimes (m_l \cdot a) + \sum_{i \in
    I_l} \; \sum_{i+1 \le j \le n} \; \left( \gamma_{i,i+1} \gamma_{i+1,i+2}
    \dots \gamma_{j-1,j} \right) \otimes \psi^{ij}_a(m) \]
for $1 \le l \le r, \; a \in A, \; m_l \in X_l$ defines a noncommutative
deformation $X_{\Gamma} \in \defm{\sfam(\Gamma)}(k[\Gamma])$.
\end{proof}

We say that $\sfam(\Gamma)$ is a swarm if $\dim_k \ext^1_A(X_i,X_j)$ if 
finite for $1 \le i,j \le r$. We shall assume that this is the case in the 
rest of this section. In this case, the noncommutative deformation functor
$\defm{\sfam(\Gamma)}$ has a miniversal object $(H,X_H)$ by general results;
see Eriksen, Laudal, Siqveland \cite{erik-laud-siqv17}, where $H$ is the
pro-representing hull in the pro-category $\cacat r$, and $X_H \in
\defm{\sfam(\Gamma)}(H)$ is the versal family. We write $X(\sfam,\Gamma) =
\mor(H,k[\Gamma])$ for the set of morphisms $\phi: H \to k[\Gamma]$ in
$\cacat r$. Note that the natural map $X(\sfam,\Gamma) \to \defm{\sfam(\Gamma)}
(k[\Gamma])$ given by $\phi \mapsto M_{\phi} = \defm{\sfam(\Gamma)}(\phi)(M_H)$ 
is surjective by the versal property.

\begin{lemma}
The set $X(\sfam,\Gamma) = \mor(H, k[\Gamma])$ is an affine algebraic variety.
\end{lemma}
\begin{proof}
Since $k[\Gamma]$ is an algebra in $\acat r(n)$, any morphism $\phi: H \to
k[\Gamma]$ in $\cacat r$ can be identified with $\phi_n: H_n \to k[\Gamma]$
since $\phi(I(H)^n) = 0$. To prove that $X(\sfam, \Gamma) = \mor(H_n,
k[\Gamma])$ is an affine algebraic variety, it is enough to notice that $H_n$
is a quotient of $\infdef_n$, that $\mor(\infdef_n, k[\Gamma])$ is isomorphic
to affine space $\aff^N$, where
    \[ N = \sum_{i,j} \, \dim_k \ext^1_A(X_i,X_j) \cdot \dim_k \left(
    I(k[\Gamma])/I(k[\Gamma])^2 \right)_{ij} \]
and that $\mor(H_n, k[\Gamma]) \subseteq \mor(\infdef_n, k[\Gamma])$ is a
closed subset in the Zariski topology, with equations given by the
obstructions $f_{ij}(l)^n \in \infdef_n$.
\end{proof}

\begin{corollary} \label{c:defm-itext-moduli}
The set $\mathcal E(\sfam, \Gamma)$ of equivalence classes of iterated
extensions of the family $\sfam$ with extension type $\Gamma$ is a quotient
of the affine algebraic variety $X(\sfam,\Gamma)$, determined by the
noncommutative deformations of $\sfam(\Gamma) = \{ X_1, \dots, X_r \}$.
\end{corollary}

\bibliographystyle{amsplain}
\bibliography{eeriksen}

\affiliationone{
   Eivind Eriksen\\
   BI Norwegian Business School, \\
   Department of Economics, \\
   N-0442 Oslo, Norway \\
   \email{eivind.eriksen@bi.no}}

\end{document}